\documentclass[11pt]{amsart}
\normalfont

\usepackage{amssymb}

\usepackage[all]{xy}

\renewcommand{\epsilon}{\varepsilon}
\renewcommand{\setminus}{\smallsetminus}

\newtheorem{theorem}{Theorem}[section]
\newtheorem{proposition}[theorem]{Proposition}
\newtheorem{corollary}[theorem]{Corollary}
\newtheorem{lemma}[theorem]{Lemma}

\newtheorem{question}[theorem]{Question}

\theoremstyle{propositionA}

\newtheorem{theoremA}{Theorem}[section]

\theoremstyle{definition}

\newtheorem{notation}[theorem]{Notation}

\theoremstyle{remark}
\newtheorem{remark}[theorem]{Remark}

\newcommand{\Q}{\mathbb Q}
\newcommand{\Z}{\mathbb Z}

\newcommand{\R}{\mathbb R}
\newcommand{\C}{\mathbb C}

\newcommand{\OO}{\mathcal O}

\newcommand{\cohom}[3]{H^{{\raise1pt\hbox{$\scriptstyle#1$}}}(#2\>\!,#3)}
\newcommand{\tatecohom}[3]%
  {\widehat H^{{\raise1pt\hbox{$\scriptstyle#1$}}}(#2\>\!,#3)}

\newcommand{\Cohom}[3]%
  {H^{{\raise1pt\hbox{$\scriptstyle#1$}}}\big(#2\>\!,#3\big)}
\newcommand{\Tatecohom}[3]%
  {\widehat H^{{\raise1pt\hbox{$\scriptstyle#1$}}}\big(#2\>\!,#3\big)}

\newcommand{\homol}[3]{H_{{\lower1pt\hbox{$\scriptstyle#1$}}}(#2\>\!,#3)}
\newcommand{\homolog}[2]{H_{{\lower1pt\hbox{$\scriptstyle#1$}}}(#2)}

\title[]{Undistorted embeddings of metabelian groups of finite Pr\"ufer rank}

\author{Sean Cleary}
\author{Conchita Mart\'inez-P\'erez}

\bigskip
\date{today}

\thanks{The second author  was supported by  Gobierno de Arag\'on, European Regional 
Development Funds,
MTM2010-19938-C03-03 and Gobierno de Arag\'on, 
Subvenci\'on de Fomento de la Movilidad de los Investigadores.}

\begin{document}
\begin{abstract} General arguments of Baumslag and Bieri guarantee that any metabelian group of finite Pr\"ufer rank can be embedded in a metabelian constructible group.  Here, we consider the metric behavior of a rich class of examples and analyze the distortions of specific embeddings.
\end{abstract}

\maketitle

\section{Introduction}

Though the set of all metabelian groups is quite varied, there are some good restrictions on metabelian groups of particular types.    The class of constructible groups are those which can be built from the trivial group using finite extensions, finite amalgamated products and finite rank HNN extensions where all the attaching subgroups involved are themselves constructible.  Such constructible groups are naturally finitely presented and further of type $FP_\infty$. 
 Baumslag and Bieri \cite{BaumslagBieri} proved that any finitely generated metabelian group of finite Pr\"ufer rank (see Section \ref{lower}) can be embedded in a metabelian constructible group, showing the richness of possible subgroups of constructible groups. Moreover, since metabelian constructible groups  have finite Pr\"ufer rank, this hypothesis can not be dropped (however, there are  some embedding results for arbitrary metabelian groups in less restrictive groups, see for example \cite{BogleyHarlander} or \cite{DesiFlavia}).

Our motivation in this paper was to find conditions that guarantee that  the Baumslag-Bieri embedding is undistorted. 
There are some results in the literature about distortion in metabelian groups defined using wreath products (see for example \cite{Sean} or \cite{Davis}) but in our case, we are looking at a very different type of metabelian group.
A possible way to attack the problem is via an estimation of geodesic words lengths in finitely generated metabelian groups of finite Pr\"ufer rank.  If the group is split (that is,  of the form $G=Q\rtimes B$ with $Q$, $B$ being abelian), this reduces to the problem of estimating $G$-geodesic word lengths for elements in $B$; moreover one can assume that $B$ lies inside a finite dimensional rational vector space.

In the polycyclic case one has $B\cong\Z^n$, in fact finitely generated, so it has its own geodesic word lengths (which, upon fixing a generating system are quasi-equivalent  to  the usual 1-norm $\|(v_1,\ldots,v_n)\|=\sum_{i=1}^n|v_i|$). Then  one can estimate the geodesic word length in the whole group in terms of $\|\cdot\|$. This is what Arzhantseva and Osin do in Lemma 4.6 of \cite{ArzhantsevaOsin} where they obtain a bound for geodesic word lengths in some metabelian polycyclic groups which they use later to show that certain embedding is undistorted.

In order to be able to do something similar but in the most general case when $B$ is isomorphic to a subgroup of $\Q^n$ (note that we do not assume $G$ finitely presented) we introduce in Section \ref{preliminaries} what we call the $\mu$-norm in $\Q$ and its extension to $\Q^n$ which we denote $\mu_E$. If $G$ is torsion-free we may also represent the action on $B$ of any element of $G/G'$ using a rational matrix. The eigenvalues of this matrix are referred to as the eigenvalues of the element acting and in a similar way we talk about semi-simple actions when the associated matrices are semi-simple. Using the norm $\mu_E$ we get our first main result:

\begin{theoremA}\label{main1} Let $G$ be finitely-generated torsion-free metabelian of finite Pr\"ufer rank and let $B\leq \hat G\leq G$ be subgroups such that
\begin{itemize}
\item[i)] $B$ is abelian,

\item[ii)] either $B=G'$ or $G'\leq B$ and $G=Q\ltimes B$,

\item[iii)] $\hat G$ is finitely generated and acts semi-simply on $B$,

\item[iv)] there is some $g\in\hat G$ acting on $B$ with no eigenvalue of complex norm 1.
\end{itemize}
Then for any $b\in B$,
$$\|b\|_G\sim\ln(\mu_E(b)+1).$$
\end{theoremA}


Using an old result in elementary number theory due to Kronecker we check that the hypothesis in Theorem \ref{main1} about the eigenvalues of the action holds true in the following special circumstances.

\begin{theoremA}\label{main1,5} Let $G$ be a  torsion-free  group with $G=\langle g\rangle\ltimes B$ with $B$ abelian. We assume that there is no finitely generated subgroup of $B$ setwise invariant under the action of $g$ and that $G$ is finitely presented. 
Then $g$ 
acts on $B$ with no eigenvalue of complex norm 1.\end{theoremA} 

Finally, using  the estimation of word lengths for elements in $B$,  we show that in the split case of the hypothesis of Theorem \ref{main1}  the Baumslag-Bieri embedding is undistorted.

\begin{theoremA}\label{main2} Let $G=Q\ltimes B$ be finitely generated of finite Pr\"ufer rank  with $Q$ and $B$ abelian. We assume that there is a subgroup $B\leq \hat G\leq G$ such that
\begin{itemize}
\item[i)] $\hat G$ is finitely generated and acts semi-simply on $B$,

\item[ii)] there is some $g\in\hat G$ acting on $B$ with no eigenvalue of complex norm 1.
\end{itemize}

Then $G$ can be embedded without distortion in a metabelian constructible group.
\end{theoremA}

The particular case when the group $G$ is a solvable Baumslag-Solitar group is part of  \cite[Theorem 1.4]{ArzhantsevaOsin}. This last result of Arzhantseva and Osin also shows that Dehn function of the group in which we are embedding is at most cubic and this was later improved to quadratic by De Cornulier and Tessera in \cite{CornulierTessera}. We have not done any attempt to look at the Dehn function of the group that one gets in the embedding but it seems plausible that the following question has an affirmative answer.

\begin{question}{\rm Let $G$ be a group satisfying the hypothesis of Theorem \ref{main2} (or some variation). Can $G$ be embedded without distortion in a metabelian constructible group with quadratic Dehn function?}
\end{question}

\section{Norms and quasi-equivalences}\label{preliminaries}

\begin{notation} Let $X$ be a set and let  $f,g:X\to\R^+$ be maps. If there are constants $C\geq 0$, $M\geq 1$, such that $f(x)\leq Mg(x)+C$ for any $x\in X$, then we put $f\preccurlyeq g$. If $f\preccurlyeq g$ and $g\preccurlyeq f$, then we denote $f\thicksim g$ and say that both functions are {\sl quasi-equivalent}. 

Let $G$ be a group generated by a finite set $\Omega$. Denote by $\parallel\cdot\parallel_G$ the geodesic word metric on $G$ with respect to $\Omega$. This is well-defined up to quasi-equivalence and since we also work up to quasi-equivalence we do not need to be very explicit about which particular generating system we are using. Let $H\leq G$ be a finitely generated subgroup and $\|\cdot\|_H$ its geodesic word metric. We say that $H$ is {\sl undistorted} in $G$ if $\|\cdot\|_G\thicksim\|\cdot\|_H$ where the first function is restricted to $H$. Note that since for any $h\in H$, $\|h\|_G\preceq\|h\|_H$, this is equivalent to $\|\cdot\|_H\preccurlyeq\|\cdot\|_G$.

As mentioned in the introduction, if $A$ is a finitely generated free abelian group then, upon fixing a basis, $A$ can be identified with $\Z^n$. Then $\|\cdot\|_A$ is quasi-equivalent to the restriction to $\Z^n$ of the usual 1-norm in $\R^n$ which we just denote $\|\cdot\|$
$$\|(a_1\ldots,a_n)\|=\sum_{i=1}^n|a_i|.$$

It will also be useful at some point to work with matrix norms, we define the in the usual way: For an $n\times n$-matrix $M$, let
$$\|M\|=\text{max}\{{\|Mv\|\over\|v\|}\mid 0\neq v\in\R^n\}.$$

To get an analogous norm but for $B\subseteq\Q^n$ we introduce what we call the $\mu$-norm:
Let $\gamma/\alpha\in\Q$. We set
$$\mu\Big({\gamma\over\alpha}\Big):={|\gamma\alpha|\over\text{gcd}(\gamma,\alpha)^2}.$$
We extend this  to $\Q^n$ in the obvious way:
 Let $\mathcal{B}=\{v_1,\ldots,v_n\}$ be a basis of $\Q^n$ and  $(\beta_1',\ldots,\beta_n')$ the coordinates of $b\in\Q^n$ in $\mathcal{B}$. We put
 $$\mu_E^{\mathcal{B}}(b)=\sum\mu(\beta_i').$$

\end{notation}

The following easy properties will be useful below:

\begin{lemma}\label{propln} Let $X$ be a set and $f,g,h:X\to\R^+\cup\{0\}$ be maps.

\begin{itemize}
\item[i)]  If $f\preccurlyeq  h,$ then $f+g\preccurlyeq  h+g$.

\item[ii)]\label{sum1} If $f\thicksim h,$ then $f+g\thicksim h+g.$

\item[iii)] For $a,b\in\R,$ $|a|,|b|\geq 1$, $\ln(|a|+|b|)\thicksim\ln|a|+\ln|b|.$

\item[iv)]\label{ln+1} For $c\in\R$, $|c|\geq 1$, $\ln|c|\sim\ln(|c|+1).$

\item[v)]\label{lnsum} For $a_1\ldots,a_n\in\Z,$
$$\ln(\sum_{i=1}^n|a_i|+1)\sim\sum_{i=1}^n\ln(|a_i|+1)$$
where the relevant constants depend on $n$.

\item[vi)] If $f\preccurlyeq  h$, then $\ln(f+1)\preccurlyeq\ln(h+1)$.

\end{itemize}
\end{lemma}
\begin{proof}

\begin{itemize}
\item[i)] Let $C\geq 0$, $M\geq 1$ with $f(x)\leq Mh(x)+C$ for any $x\in X$. Then
$$f(x)+g(x)\leq Mh(x)+Mg(x)+C.$$

\item[ii)] Follows from i).

\item[iii)] We have 
$${1\over 2}\ln|a|+{1\over 2}\ln|b|+\ln(2)\leq\ln(|a|+|b|)\leq\ln|a|+\ln|b|+\ln(2)$$

\item[iv)]  This is a particular case of iv).

\item[v)]  Use induction, iv) and v).

\item[vi)] By i) (or by iv)) it suffices to show that $\ln(f)\preccurlyeq\ln(h)$ if $f\preccurlyeq  h$ and $f(x),h(x)\geq 1$ for any $x\in\R$. Let $C\geq 0$, $M\geq 1$  with $f(x)\leq Mg(x)+C$ for any $x\in X$, note that we may assume $C\geq 1$. Taking $\ln$ and using iii)
$$\ln(f)\leq\ln(Mg+C)\thicksim\ln(Mg)+\ln (C)=\ln(g)+\ln M+\ln C.$$

\end{itemize}
\end{proof}

\subsection{Properties of $\|\cdot\|$} The norm $\|\cdot\|$ in $\R^n$ is well known to be quasi-invariant up to change of basis. As a consequence, if we have a splitting $\R^n=U_1\oplus U_2\oplus\ldots\oplus U_s,$ then for any family of vectors $u_i\in U_i$, $i=1,\ldots,s$ 
\begin{equation}\label{sum}
\|u_1\|+\ldots+\|u_s\|\sim\|u_1+\ldots+u_s\|,
\end{equation} 
where the relevant constants depend on the particular subspaces $U_1,\ldots,U_s$.

On the other hand, it follows from Lemma \ref{lnsum} that for $a\in\Z^n$,
$$\ln(\|a\|+1)\sim\sum_{i=1}^n\ln(|a_i|+1).$$

\subsection{Properties of $\mu_E$}
In general, the function $\mu$ behaves badly with respect to sums, and this has the unpleasant consequence that it is not quasi-invariant upon change of basis. However, things get better with logarithms, thanks to the following property:

\begin{lemma}\label{mu2} For $\beta_1,\beta_2\in\Q$ arbitrary,
 $$\ln\Big(\mu(\beta_1+\beta_2)+1\Big)\preceq\ln\Big(\mu(\beta_1)+1\Big)+\ln\Big(\mu(\beta_2)+1\Big).$$ 

\end{lemma}
\begin{proof} We may assume that $0\neq\beta_1,\beta_2,\beta_1+\beta_2$ thus since all the numbers inside $\ln$ are integers by Lemma \ref{lnsum} it suffices to check that
$$\ln\Big(\mu(\beta_1+\beta_2)\Big)\preceq\ln\Big(\mu(\beta_1)\Big)+\ln\Big(\mu(\beta_2)\Big).$$
We put $\beta_i=k_i/m_i$ with $k_i$ and $m_i$ coprime integers for both $i=1,2$.
Let $d$ be the greatest common divisor of $m_1$ and $m_2$ and we put $m_i=m_i'd$ again for $i=1,2$. Then 
$$\begin{aligned}\ln\Big(\mu(\beta_1+\beta_2)\Big)=\ln(\mu\Big({k_1m_2'+k_2m_1'\over m_1'm_2'd}\Big))\leq\ln|(k_1m_2'+k_2m_1')m_1'm_2'd|\\
=\ln|k_1m_1(m_2')^2+k_2m_2(m_1')^2|\\
\preceq\ln|k_1m_1|+2\ln|m_2'|+\ln|k_2m_2|+2\ln|m_1'|\\
\leq 3\ln\Big(\mu(\beta_1)\Big)+3\ln\Big(\mu(\beta_2)\Big).
\end{aligned}$$
where we use  Lemma \ref{ln+1} to split the sum inside the logarithm.
\end{proof}

\begin{lemma}\label{inv} With the notation above,
$$\ln(\mu_E^{\mathcal{B}}(b)+1)\sim\ln(\mu_E(b)+1).$$
\end{lemma}
\begin{proof} Note that by Lemma \ref{ln+1} it suffices to show that for $0\neq b\in\Q^n$,
$$\ln(\mu_E^{\mathcal{B}}(b))\preceq\ln(\mu_E(b)).$$
The coordinates of $b=(\beta_1,\ldots,\beta_n)$ in $\mathcal{B}$ can be obtained from $(\beta_1,\ldots,\beta_n)$ by a successive application of one of these operations:
\begin{itemize}
\item[i)] $\beta_i\mapsto\beta_i+\beta_j$ for some $i,j$,

\item[ii)] $\beta_i\mapsto\tau\beta_i$ for some $i$ and certain $0\neq\tau\in\Q$ belonging to certain set depending on $\mathcal{B}$ only.
\end{itemize}
Any $\tau$ as in ii) belongs to a prefixed set so we have 
$$\mu(\tau\beta)\leq\mu(\tau)\mu(\beta)\sim\mu(\beta).$$
So we only have to check that for $0\neq\beta_1,\beta_2,\beta_1+\beta_2$,
$$\ln\mu(\beta_1+\beta_2)+\ln\mu(\beta_2)\preceq\ln\mu(\beta_1)+\ln\mu(\beta_2).$$
But this is an obvious consequence of Lemma \ref{mu2}.

\end{proof}

As a consequence, we have the $\mu_E$-version of the sum formula (\ref{sum}) but after taking logarithms:
 Let $\Q^n=U_1\oplus\ldots\oplus U_s$ and for $i=1,\ldots,s$ fix an isomorphism $U_i\to\Q^{\dim U_i}$. Let $\mu_E^{U_i}$ denote the obvious version of the $\mu_E$-norm in $U_i$ defined using this isomorphism. Then for $u_i\in U_i$, $i=1,\ldots,s$
\begin{equation}\label{musum} 
\ln\Big(\mu_E(u_1+\ldots+u_s)+1\Big)\sim\ln\Big(\mu^{U_1}_E(u_1)+1\Big)+\ldots+\ln\Big(\mu^{U_s}_E(u_2)+1\Big)\end{equation}
and in particular for $U=U_1$ and $u\in U$,
\begin{equation}\label{muembed} 
\ln(\mu_E(u)+1)\sim\ln(\mu^U_E(u)+1).
\end{equation}
(here the relevant constants depend on the particular subspaces $U_1,\ldots,U_s$). We will omit in the sequel the over indices $U_i$ from the notation.

\section{The lower bound}\label{lower}


\begin{notation}\label{genmetabelian} Recall that a group is said of {\sl finite Pr\"ufer rank} if there is a bound on the number of elements needed to generate any finitely generated subgroup. We let $G$ be a finitely generated group of finite Pr\"ufer rank with a normal subgroup $B$ such that $B$ and $G/B$ are abelian. Then by \cite[Proposition 1.2]{Boler}, the torsion subgroup $T(B)$ of $B$ is finite. Therefore, in order to estimate word lengths in $B$, we can just pass to the torsion-free group $B/T(B)$. Something analogous happens with the group $G/B$. However to avoid unnecessary complications, we will just assume for now that both $B$ and $Q=G/B$ are torsion-free.  Let $k$ be the rank of $Q$ and $n=\text{dim}_\Q B\otimes \Q$. We choose $a_1,\ldots,a_n\in B$ to be a maximal linearly independent family generating $B$ as $G$-module and $q_1,\ldots,q_k\in G$ such that $\{q_1B,\ldots,q_kB\}$ is a generating system for $Q$. Then
$$\mathcal{X}=\{q_i^{\pm1},a_j^{\pm 1}\}_{1\leq i\leq k,1\leq j\leq n}$$ is a generating system for $G$. 
Using this generating system, we get an embedding $\varphi:B\to\Q^n$ so that for $A=\langle a_1,\ldots,a_n\rangle$, $\varphi(A)=\Z^n$. We use $\varphi$ to extend the definition of $\mu_E$ to $A$ and  of the  norm $\|\cdot\|$ to $A$. Explicitly, we set, for $b\in B$, $a\in A$:
$$\mu_E(b):=\mu_E(\varphi(b)),$$
$$\|a\|:=\|\varphi(a)\|.$$
In particular in the case when $n=1$ we have $\mu(b):=\mu(\varphi(b))$ and $|a|:=|\varphi(a)|$. By Lemma \ref{inv}, the function $\ln(\mu_E(b)+1)$ is independent, up to quasi-equivalence, on the choice of $\{a_1,\ldots,a_n\}.$ The same fact for $\|\cdot\|$ is well known.
\end{notation}

We are going to  get a bound for word lengths in $G$ by means of the collecting process used by Arzhantseva and Osin in \cite[Lemma 4.9]{ArzhantsevaOsin}.

\begin{theorem}\label{easyineq} Let $G=Q\ltimes B$ be a finitely generated torsion-free group with $Q$  and $B$ abelian of finite Pr\"ufer rank. With notation \ref{genmetabelian} and for any $b\in B$,
$$\ln(\mu_E(b)+1)\preccurlyeq\|b\|_G.$$
\end{theorem}
\begin{proof} (Collecting process). 
As mentioned in the introduction by the embedding result of Baumslag and Bieri  (\cite{BaumslagBieri}), $G$ embeds in a constructible metabelian group  $G_1=Q_1\ltimes B_1$ so that $B$ embeds in $B_1$, $Q$ embeds into $Q_1$, $B_1$ is abelian of the same Pr\"ufer rank as $B$ and  $Q_1$ is free abelian of finite rank.  Moreover, we may choose a basis of $Q_1$ such that the action of each basis element in $B_1$ is encoded by an integer matrix. As $\|b\|_{G_1}\preccurlyeq\|b\|_G$ for any $b\in B$, we may assume that $G$ itself fulfills these conditions.

 Recall that we use Notation \ref{genmetabelian}.   For each $1\leq i\leq k$, let $M_i$ be the integer matrix representing the action of $q_i$ on $B$ ($M_i$ depends on $\varphi$). All those matrices commute pairwise. Note that the $j$-th column of $M_i$ determines $a_j^{q_i}$ as linear combination of $\{a_1,\ldots,a_n\}$.  Let $n_i=\text{det}M_i$, then $n_i\in\Z$ and there is some integer matrix $N_i$ with $M_i^{-1}={1\over n_i}N_i$.
 
For a word $w$ on $\mathcal{X}$ put 
$$\begin{aligned}\sigma_i(w)&:=\text{ number of instances of $q_i$ in $w$},\\
\alpha_j(w)&:=\text{ number of instances of $a_j^{\pm 1}$ in $w$}.\end{aligned}$$
 Then $$\text{length}(w)=\sum_{j=1}^n\alpha_j(w)+2\sum_{i=1}^k\sigma_i(w).$$
 We let $b\in B$, and note that we may assume that $b\neq 1$ and show the result but for $\ln(\mu_E(b))$. We choose a geodesic word  $w_b$  representing $b$, therefore  $\|b\|_G=\text{length}(w_b)$.
We claim that there is some word $w_c$ in $a_1^{\pm},\ldots,a_n^{\pm}$ only
and some $q\in Q$ such that
$$w_b=_Gq^{-1}w_cq.$$
To see that, we only have to use the equations $a_jq_i=q_ja_j^{q_i}$ and $q_i^{-1}a_j=a_j^{q_i}q_i^{-1}$ to move all the instances of $q_i$ with positive exponent to the left and all the instances with negative exponent to the right. In this process we do not cancel $q_i$ with $q_i^{-1}$; we just apply the commutativity relations in $Q$. Since these equations do not change the number of $q_i$'s, we have
$$q=\prod_{i=1}^k q_i^{\sigma_i(w)}.$$
Let $c\in A$ be the element represented by $w_c$. Let
$$n_q=\prod_{i=1}^k n_i^{\sigma_i(w)},\,
N_q=\prod_{i=1}^k N_i^{\sigma_i(w)}$$ 
and $M_q^{-1}={1\over n_q}N_q$.
With additive notation we have $b=M_q^{-1}c={1\over n_q}N_qc$. Put
$$K_1=\text{max}\{\|N_i\|\mid 1\leq i\leq k\},$$
$$K_2=\text{max}\{\|M_i\|\mid 1\leq i\leq k\},$$
$$K_3=\text{max}\{n_i\mid 1\leq i\leq k\}.$$
Observe that $K_1$, $K_2$ and $K_3$ are positive integer numbers and 
$$\|N_q\|\leq K_1^{\sum_{i=1}^k\sigma_i(w)}.$$
Now, let us look at how is $c$ constructed. Multiplicatively, it is the product of elements of the form
$(a_j^{\pm 1})^{z}$ where $j=1,\ldots,n$, $a_j^{\pm1}$ is one of the $\alpha_j(w)$ instances of $a_j^{\pm1}$ in $w$ and $z$ is a certain product of the $q_i$'s all of them with positive exponent. Again with additive notation this corresponds to certain product of some of the matrices $M_i$ with the vector having all its coordinates zero but for a single 1. The norm of such vector is obviously bounded by $K_2^2\sum_{i=1}^k\sigma_i(w)$. So we get
$$\|c\|\leq K_2^{2\sum_{i=1}^k\sigma_i(w)}(\sum_{j=1}^n\alpha_j(w))$$
and
$$\mu_E(b)\leq n_q\|N_qc\|\leq n_q\|N_q\|\|c\|\leq (K_1K_2^2K_3)^{\sum_{i=1}^k\sigma_i(w)}(\sum_{j=1}^n\alpha_j(w)).$$
Thus 
$$\begin{aligned}
\ln(\mu_E(b))\leq \ln(K_1K_2^2K_3)\sum_{i=1}^k\sigma_i(w)+\ln(\sum_{j=1}^n\alpha_j(w))\preceq \\
2\sum_{i=1}^k\sigma_i(w)+\sum_{j=1}^n\alpha_j(w)=\|b\|_G.\end{aligned}$$

\end{proof}

\begin{corollary}\label{easyineq1} Let $G$ be torsion-free metabelian of finite Pr\"ufer rank. Then for any $b\in G'$,
$$\ln(\mu_E(b)+1)\preccurlyeq\|b\|_G$$
where $\mu_E(b)$ is the $\mu$-norm obtained from some embedding of $G'$ into $\Q^n$ for $n=\dim_\Q G'\otimes\Q$.
\end{corollary}
\begin{proof} Using \cite{Baumslag} and \cite{Boler} one can construct an embedding $\psi:G\to G_1$ for $G_1$ a split metabelian group of the form $G_1=G/G'\ltimes B$ with $B$ abelian of finite Pr\"ufer rank so that $\psi(G')\leq B$. As $G/G'$ might contain torsion, $G_1$ might not be torsion-free but obviously it has a finite-index normal subgroup which is torsion-free and that satisfies the hypothesis of  Theorem \ref{easyineq}. Therefore Theorem \ref{easyineq} implies that for any $b\in G'$ 
$$\ln(\mu_E(\psi(b))+1)\preccurlyeq\|\psi(b)\|_{G_1}\leq\|b\|_{G}.$$
By Lemma \ref{inv}, $\ln(\mu_E(\psi(b))+1)\sim\ln(\mu_E(b)+1)$. 
\end{proof}

\section{The upper bound}\label{upper bound}

\noindent We assume that $G$ is  torsion-free finitely generated of finite Pr\"ufer rank with a  normal subgroup $B$ such that $B$ and $G/B$ are normal. The objective of this Section is to show that if moreover $G$ acts semisimply on $B$ and  there is some $g\in G$ acting on $B$ with no eigenvalue of complex norm 1 (recall that by eigenvalues of this action we mean eigenvalues of the rational matrix representing the action of $g$), then   the opposite inequality to Theorem \ref{easyineq} holds true; that is, for any $b\in B$ 
\begin{equation}\label{upperbound}\|b\|_G\preceq\ln(\mu_E(b)+1)\end{equation}
 (here, $\mu_E$ is defined using the generating system $\mathcal{X}$ of Notation \ref{genmetabelian}). 
To show this we can not use the strategy that worked fine in Section \ref{lower}, namely to embed our group in a larger, nicer group. But we can do the opposite: reduce the problem to a suitable subgroup $G_1\geq B$ since then $\|b\|_{G}\preceq\|b\|_{G_1}$. For example, we may go down a finite-index subgroup and assume that $Q=G/B$ is torsion-free (in this case, in fact, one has $\|b\|_G\thicksim\|b\|_{G_1}$ for $b\in B$).

\subsection{Reduction to abelian-by-cyclic groups.} 

We are going to use Bieri-Strebel invariants in the proof of the next result. Recall that the Bieri-Strebel invariant of a finitely generated group $G$ is a subset denoted $\Sigma_1(G)$ of the set $S(G)$ of group homomorphisms from $G$ to the additive group $\R$ having the following property: for any subgroup $H$ with $G'\leq H\leq G$, $H$ is finitely generated if and only if $\chi(H)\neq 0$ for any $\chi\in \Sigma^c_1(G)=S(G)\setminus \Sigma_1(G)$ (see Bieri, Neumann and Strebel \cite[Theorem B]{BieriNeumannStrebel}). The reader is referred to Bieri, Neumann and Strebel \cite{BieriNeumannStrebel} and  Bieri and Renz \cite{BieriRenz}  for notation and further properties of these invariants.

\begin{proposition}\label{reduction} Let $G$ be a finitely generated torsion-free metabelian group of finite Pr\"ufer rank with $G'\leq B\leq G$ abelian. Then there is some $t\in G$ such that the group $H=\langle t,B\rangle$ is finitely generated. If there is some $g\in G$ acting on $B$ with no eigenvalue of complex norm 1, then $t$ can be chosen having the same property.
\end{proposition}
\begin{proof} As $G$ has finite Pr\"ufer rank,  by \cite[Theorem 2.5]{meinert2} the complement of the Bieri-Strebel invariant $\Sigma_1(G)^c$ is finite.
$$\{\chi_1,\ldots,\chi_r\}=\{\chi\in\Sigma_1^c(G)\mid\chi(B)=0\}.$$
If $r=0$, then $B$ itself would be finitely generated and we may choose any $t\not\in B$ so we assume $r\neq 0$. For an arbitrary  $t\in G$, the subgroup $H=\langle t,B\rangle$ is finitely generated if and only if $\chi_i(t)\neq 0$ for $i=1,\ldots,r.$ So it suffices to choose any $t\in G$ with $t\not\in\cup_{i=1}^k\text{Ker}\chi_i$, this can be done as each $\text{Ker}\chi_i/B$ is a subgroup of $G/B$ of co-rank at least 1. 

Now, assume that there is some $g\in G$ acting on $B$ with no eigenvalues of complex norm 1 (in particular, this implies $t\not\in B$). Let $t$ be the element chosen above. We claim that there is some integer $n$ such that $t_1:=t^{-n}g$ has the desired properties. Note that
$\chi(t_1)=\chi_i(t^{-n}g)=-n\chi_i(t)+\chi_i(g)$
thus $n\neq \chi_i(g)/\chi_i(t)$ implies $\chi_i(t_1)\neq 0$.
 
Also, as $[g,t]\in B$, the matrices representing the action of $g$ and $t$ on $B$ commute thus the eigenvalues of $t_1$ are of the form $\lambda_1^{-n}\lambda_2$ for $\lambda_1$ eigenvalue of $t$ and $\lambda_2$ eigenvalue of $g$. As $\lambda_2\bar\lambda_2\neq 1$, we see that for all except at most finitely many values of $n$, the eigenvalues of $t_1$ also have complex norm $\neq 1$. So we only have to choose an $n$ not in that finite set and distinct from $\chi_i(g)/\chi_i(t)$ for $i=1,\ldots,r$ and we get the claim. 
\end{proof}

As $\|b\|_G\leq\|b\|_H$ for any $b\in B$, 
 from now on we may assume the following:
\begin{equation}\label{hypo3}\Big\{\begin{aligned}G=\langle t\rangle \ltimes B\text{ torsion-free  with $B\subseteq\Q^n$ and $t$ acting on $B$ via}\\
\text{a semisimple matrix with no eigenvalue of complex norm 1}.\end{aligned}\end{equation}

\subsection{The case when $B$ has Pr\"ufer rank 1} In this case, we assume
$B\subseteq_\varphi\Q$ and $A=_\varphi\Z$ (but we note that the precise identification depends on the chosen generating system  $\{t^{\pm1},a^{\pm1}\}$ ). Then $t$ acts on $B$ by multiplication by some  $0\neq\lambda=x/y\in\Q$ and up to swapping $t$ and $t^{-1}$ we may assume $|\lambda|>1$. The group $G$ is finitely presented if and only if  $\lambda\in\Z$, but note that we are not assuming that. Note also that $B=_\varphi\Z[\lambda^{\pm 1}]$.

\begin{proposition}\label{rank1A} Let $\lambda\in\Q$ and  $G=\langle t\rangle \ltimes B$ with $\varphi:B\to\Z[\lambda^{\pm 1}]$  an isomorphism and $t$ acting on $B$ via $\varphi$ and multiplication with $\lambda$. Then for $a\in A=\varphi^{-1}(\Z)$, 
$$\|a\|_G\preccurlyeq\ln(|a|+1)$$
where the relevant constants depend on $\lambda$.
\end{proposition}
\begin{proof} By going down to a finite-index subgroup if necessary, we may assume that $\lambda=x/y$ with coprime integers $x,y>0$. 
  We claim that there is a minimal $\lambda$-adic expression of any $m\in\Z$, i.e., that there is $s\geq 0$ such that
$$m=\pm(\sum_{i=0}^ sc_{i}\lambda^{i})$$
with $c_i\in\Z$ and $0\leq c_i<x$ for each $0\leq i\leq s$. To see it, we assume $m>0$ and use the Euclidean Algorithm to determine $c,r$ with $m=xc+r$ such that $0\leq r<x$ and $0\leq c<m$. Then we have
$$m=\lambda(cy)+r$$
and $\lambda cy=m-r\leq m$ thus $ cy<m$. This means that we can repeat the process and obtain an expression as desired.
Now,  
$$\lambda^{s}\leq c_{0}+c_{1}\lambda+\ldots c_s\lambda^{s}=|m|$$ 
thus 
$$s\ln\lambda\leq\ln(|m|+1)=\ln(|a|+1).$$
On the other hand, writing the $\lambda$-adic expression of $a=a_1^m$ multiplicatively using the generating system $\{t^{\pm1},a_1^{\pm1}\}$ we get
$$a=a_1^m=a_1^{c_0}t^{-1}a_1^{c_1}t^{-1}\ldots t^{-1}a_1^{c_s}t^s$$
so
 $$
\|a\|_G\leq 2s+\sum_{i=0}^s c_{i}<s(2+x)\preceq\ln(|a|+1).$$
\end{proof}

\begin{lemma}\label{Bezout}(bounded Bezout's identity) Let $0\neq x_1,y_1\in\Z$ coprime integers. Then for any $d\in\Z$ there are $a,c\in\Z$ such that $|a|$, $|c|\leq|d||x_1||y_1|$ and
$$d=x_1a+y_1c.$$
\end{lemma}
\begin{proof} Obviously it suffices to consider the case $d=1$ and we may assume $x_1,y_1\neq\pm 1$.  From Brunault \cite{Brunault} in an answer to a question by Denis Osin on mathoverflow:

We take first $a_1,c_1\in\Z$ arbitrary with $1=x_1a_1+y_1c_1$.
We let $b,c\in\Z$ with $c_1=x_1b+c$ and $0< c<|x_1|$. Then putting $a=a_1+y_1b$,
$$1=x_1a_1+y_1c_1=x_1a_1+x_1y_1b+y_1c=x_1a+y_1c.$$
 We have $x_1a=1-y_1c$ thus $a={1\over x_1}-y_1{c\over x_1}$ and
$$|a|\leq |{1\over x_1}|+|y_1||{c\over x_1}|<|{1\over x_1}|+|y_1|$$
and as $a,y_1\in\Z$ we deduce $|a|\leq|y_1|$.
\end{proof}

\begin{proposition}\label{rank1B} Let $\lambda\in\Q$ and  $G=\langle t\rangle \ltimes B$ with $\varphi:B\to\Z[\lambda^{\pm 1}]$  an isomorphism and $t$ acting on $B$ via $\varphi$ and multiplication with $\lambda$. Then for $b\in B$, 
$$\|b\|_G\preccurlyeq\ln(\mu(b)+1)$$
where the relevant constants depend on $\lambda$.
\end{proposition}
\begin{proof} Let $B_k^s=_\varphi\lambda^{-s}\Z+\lambda^{-s+1}\Z+\lambda\Z+\ldots+\lambda^{k-1}\Z+\lambda^k\Z$, then 
$$B=\bigcup_{0\leq s,k\in\Z}B_k^s.$$
For any $1\neq b\in B$ we may choose $0\leq s,k$  smallest possible so that $b\in B_k^s$. Then $\varphi(b)\in{1\over x^sy^k}\Z$ and using the classical version of Bezout's identity it is easy to see that $s,k$ are also the smallest possible non negative integers satisfying this.
So we have
$$\varphi(b)={d\over x^sy^k}$$
for $0\neq d\in\Z$ such that $x,y\nmid d$. Now take $d_1,d_2\in\Z$ such that 
$$d=x^sd_1+y^kd_2$$
and $|d_1|,|d_2|\leq|d|(|x|^s|y|^k),$ we may find them using  Lemma \ref{Bezout}. We have
$$b={d\over x^sy^k}={d_1\over y^k}+{d_2\over x^s}.$$
Using again Lemma \ref{Bezout}, we  find $r_1,c_1,r_2,c_2$ such that
$$d_1=x^kr_1+y^kc_1,$$
$$d_2=y^sr_2+x^sc_2$$
and $|r_1|,|c_1|\leq|d_1|(|x|^k|y|^k),$ $|r_2|,|c_2|\leq|d_2|(|x|^s|y|^s).$ Now
$$\varphi(b)={d_1\over y^k}+{d_2\over x^s}=\lambda^kr_1+c_1+c_2+\lambda^{-s}r_2.$$
Multiplicatively $b=(a_1^{r_1})^{t^k}a_1^{c_1}a_1^{c_2}(a_1^{r_2})^{t^{-s}}$ which yields
$$\|b\|_G\leq 2k+2s+\|a_1^{r_1}\|_G+\|a_1^{c_1}\|_G+\|a_1^{c_2}\|_G+\|a_1^{r_2}\|_G.$$
Consider for example $r_1$ and assume $r_1\neq 0$. By Proposition \ref{rank1A} 
$$\begin{aligned}\|a_1^{r_1}\|_G\preceq\ln(|r_1|)\preceq\ln|d_1|+k(\ln|x|+\ln|y|)\leq\\ \ln|d|+(s+k)\ln|x|+2k\ln|y|\leq\\
 (s+k)3(\ln|x|+\ln|y|)+\ln|d|\preceq\\
  s+k+\ln|d|.\end{aligned}$$
The same holds for $c_1,c_2,r_2$ so we get
$$\|b\|_G\preceq s+k+\ln|d|.$$
Now we let $l=\text{gcd}(x^sy^k,d)$. Obviously, $l\leq x^sy^k$  thus
$$s+k+\ln|d|=s+k+\ln|l|+\ln|d/l|\leq s+k+\ln(x^sy^k)+\ln|d/l|\preceq s+k+\ln|d/l|$$
Finally, as $x,y\nmid d$, there is some prime $p\mid x$ with $p\nmid d$ and some prime $q\mid y$ with $q\nmid d$. Then $p^sq^k\mid(x^sy^k/l)$ thus
$$s+k\preceq\ln|{x^sy^k\over l}|$$
and we get
$$s+k+\ln|d|\preceq s+k+\ln|{d\over l}|\preceq\ln|{x^sy^k\over l}|+\ln|{d\over l}|=\ln|{x^sy^kd\over l^2}|=\ln(\mu(b)).$$
\end{proof}

\subsection{The lower bound for elements in $A$} We proceed now to the proof of the upper bound in the case when the Pr\"ufer rank of $B$ is an arbitrary positive integer. We follow the same strategy as before: we check it first for elements in $A$, then for any $b\in B$.

\begin{proposition}\label{rankarbA} Let $G=\langle t\rangle \ltimes B$ finitely generated with a monomorphism $\varphi:B\hookrightarrow\Q^n$ and $t$ acting on $B$ via $\varphi$ and
a semisimple matrix with no eigenvalue of complex norm 1. Then for any $a\in A=\varphi^{-1}(\Z^n)$ we have
$$\|a\|_G\preccurlyeq\ln(\|a\|+1).$$
\end{proposition}
\begin{proof} Let $M$ be the matrix encoding the action of $t$ on $B$. We let $M$ act on the vector space $\R^n$.
As $M$ is semisimple, we may split $\R^n=U\oplus V$ with $U$ and $V$ both $M$-closed and so that for any $u\in U$, $\|Mu\|>\|u\|$, whereas for any $v\in V$, $\|Mv\|<\|v\|$. To see it, we let $U=U_1\oplus\ldots\oplus U_m$ with each $U_i$ either a 1-dimensional eigenspace associated to a real eigenvalue $|\lambda_i|>1$ or a 
2-dimensional $M$-closed real space associated to a pair of complex eigenvalues $\lambda_i,\bar\lambda_i$ with $\lambda_i\bar\lambda_i>1$. We do the same for $V$ with those eigenvalues such that $\lambda_i\bar\lambda_i<1$. If $U_i$ is a subspace of the first kind associated to a real eigenvalue, then $\|M^{-1}u_i\|=|\lambda_i|^{-1}\|u_i\|$
for any $u_i\in U_i$. And if $U_i$ is associated to a pair of complexes eigenvalues $\lambda_i,\bar\lambda_i$ then for any $u_i\in U_i$ we may find complex vectors $v_i,\bar v_i$ with $Mv_i=\lambda_iv_i$ and $u_i=v_i+\bar v_i$. Then $M^{-k}u_i=\lambda_i^{-k}v_i+(\bar \lambda_i)^{-k}\bar v_i$ is just half of the real part of $\lambda^{-k}v_i$. This means that 
for $u\in U$
$$\text{lim}_{k\to\infty}\|M^{-k}u\|=0$$
and analogously one sees that for $v\in V$,
$$\text{lim}_{k\to\infty}\|M^{k}v\|=0.$$

We put
$$\varphi(a)=u+v$$
with $u\in U$, $v\in V$. We claim that there is a bound $K$ (not depending on $a$) such that for some $k,l\geq 0$ and $r_i,p,s_j\in\Z^n$ with $\|r_j\|,\|p\|,\|s_j\|\leq K$,
\begin{equation}\label{claim}\varphi(a)=M^kr_k+M^{k-1}r_{k-1}+\ldots+Mr_1+p+M^{-1}s_{1}+\ldots+M^{-l}s_{l}.\end{equation}
(In particular, there are only finitely many possible choices for $r_j,p,s_j$). 

There is an integer $m\in\Z$ such that both $mM$ and $mM^{-1}$ are integer matrices, we fix $m$ from now on. The set $m\Z^n$ is a lattice in $\R^n$.
For any $j\geq 0$ let $u_j\in m\Z^n$ be as close as possible (with respect to the norm $\|\cdot\|$) to $M^{-j}u$. We have $\|u_j-M^{-j}u\|\leq {mn/2}$. If $u_0=0$, let $k=0$, in other case take $k$ such that $u_{k+1}=0$, $u_k\neq 0$. 
  Let $r_j:=u_j-Mu_{j+1}$ for $0\leq j\leq k$, then $r_j\in\Z^n$ and
$$\begin{aligned}
u_0=-M^{k+1}u_{k+1}+M^ku_k-M^ku_k+M^{k-1}u_{k-1}-\ldots+Mu_1-Mu_1+u_0=&\\
=M^kr_k+M^{k-1}r_{k-1}+\ldots+Mr_1+r_0.&\\
\end{aligned}$$
Now
$$r_j=u_j-Mu_{j+1}=(u_j-M^{-j}u)-M(u_{j+1}-M^{-j-1}u)$$
 thus
 $$\|r_j\|\leq mn.$$
 Next, we repeat the process with $v$ but this time we choose elements $v_j\in m\Z$ that approximate $M^jv$. If $v_0=0$, put $l=0$, in other case we get
$$v_0=M^{-l}s_l+M^{-l+1}s_{l-1}+\ldots+M^{-1}s_1+s_0$$
with $\|s_j\|\leq mn$. Finally, note that by construction for
$w:=u-u_0,$ $w':=v-v_0$ we have $\|w\|,\|w'\|\leq nm/2$ and
$$\varphi(a)=u+v=u_0+v_0+w+w'.$$
Let $p=r_0+s_0+w+w'$. The equation above implies that $p\in\Z^n$ and since $\|p\|\leq 3mn$ we get the claim with $K=3nm.$

Now, we may use (\ref{claim}) to write a word in $G$ with multiplicative notation representing $a$. Every instance of $M^jr_j$ with $r_j=(r_{j,1},\ldots,r_{j,n})$ will be substituted by $t^{-j}a_1^{r_{j,1}}\ldots a_n^{r_{j,n}}t^j$ and the same for the $s_j$'s. This means that if we put
$$C=\text{max}\{\|\varphi^{-1}(c)\|_G\mid c\in\Z^n,\|c\|\leq K\}$$ 
(this is a finite set, so $C<\infty$), then we have
$$\|a\|_G\leq 2(k+l)+C(k+l+1)=(2+C)(k+l)+C\preceq k+l.$$

Next, we are going to prove that 
 \begin{equation}\label{boundk}k\preceq\ln(\|u\|+1).\end{equation} 
 Obviously, we may assume $k\neq 0$.  Let $\nu:=M^{-k}u$, then $u=M^k\nu$ and $\nu\in U$. Using the splitting $U=U_1\oplus\ldots\oplus U_m$ we get $\nu=\nu_1+\ldots+\nu_m$ with $\nu_i\in U_i$.
We observe that the closest element to $\nu$ in $m\Z^n$ is $u_k$ and that $u_k\neq 0$. This implies $\|\nu\|\geq m/2$.  Take $\lambda$ be the eigenvalue $\lambda_i$ of smallest complex modulus, (which we denote $|\cdot|$). Then
$${m\over 2}|\lambda|^k\leq|\lambda|^k\|\nu\|\preceq|\lambda|^k\sum_{i=1}^m\|\nu_i\|\leq\sum_{i=1}^m|\lambda_i|^k\|\nu_i\|$$

 For each $i=1,\ldots,m$, if $\lambda_i$ is a real eigenvalue, we have $\|M^k\nu_i\|=|\lambda_i|^k\|\nu_i\|$. Otherwise, we let $\|\cdot\|_2$ denote the Euclidean norm, then, choosing a suitable basis of $U_i$
 $$|\lambda_i|^k\|\nu_i\|\preceq|\lambda_i|^k\|\nu_i\|_2=\|M^k\nu_i\|_2\preceq\|M^k\nu_i\|.$$
  Therefore
 $${m\over 2}|\lambda|^k\leq\sum_{i=1}^m\|\lambda_i\|^k\|\nu_i\|=\sum_{i=1}^m\|M^k\nu_i\|\preccurlyeq\|M^k\nu\|=\|u\|$$
thus $k\preccurlyeq\ln(\|u\|+1)$.
 One can prove that $l\preccurlyeq\ln(\|v\|+1)$ by proceeding analogously using $M^{-1}$ and $v$, $V$ instead.

Finally all this implies the lower bound $$k+l\preceq \ln(\|u\|+1)+\ln(\|v\|+1)\sim\ln(\|u\|+\|v\|)\sim\ln(\|a\|+1).$$

\end{proof}

\goodbreak

\subsection{The lower bound for elements in $B$: reductions}\label{reductions} 

We will show now that we may reduce the problem of bounding above  word lengths for elements in $b$ to a special situation. Recall that we have
$G=\langle t\rangle \ltimes B$ finitely generated with  $\varphi:B\to\Q^n$ injective so that $t$ acts semi-simply on $B$ and $A\leq B$ free abelian of rank $n$.

\subsubsection{We may assume that $B\otimes\Q$ is simple as $\Q\langle t\rangle$-module.}
The fact that $t$ acts semi-simply implies that $B\otimes\Q$ can be split as a sum of irreducible $\Q\langle t\rangle$-modules. If we had $B\otimes\Q=W_1\oplus W_2$ with $W_1,W_2\neq 0$ then we could choose $A_1\subseteq W_1$ and $A_2\subseteq W_2$ subabelian groups of the maximal possible rank and so that $A\subseteq A_1\oplus A_2$. Let $B_i=\Z\langle t\rangle A_i$ for $i=1,2$. We have $B_i\subseteq W_i$ and $B\subseteq B_1\oplus B_2$. 
Set $\tilde G=\langle t\rangle\ltimes (B_1\oplus B_2)$, obviously $G\leq\tilde G$. Moreover as both a free abelian groups of the same (finite) rank, the index of $A$ in $A_1\oplus A_2$ is finite. Then one easily checks that the index of $B$ in $B_1\oplus B_2$ and therefore also the index of $G$ in $\tilde G$ is finite.

This means that we may assume that $G=\tilde G$, equivalently that
 $B=B_1\oplus B_2$ with $B_1$ and $B_2$ $t$-invariant.  We put $G_1=\langle t \rangle\ltimes B_1$ and $G_2=\langle t \rangle\ltimes B_2$. For $b\in B$, let $b_1\in B_1$ and $b_2\in B_2$ be the elements with $b=b_1+b_2$. Then 
$$\|b\|_G\preceq\|b_1\|_{G_1}+\|b_2\|_{G_2}$$
and
$$\ln(\mu_E(b_1)+1)+\ln(\mu_E(b_2)+1)\sim\ln(\mu_E(b)+1).$$
Therefore if we had 
$$\|b_i\|_{G_i}\preceq\ln(\mu_E(b_i)+1)$$
for $i=1,2$, the we would deduce the same for $G$ and $b$.

\subsubsection{We may assume $B=\OO_L[\lambda^{\pm 1}]$ and $A=\OO_L$ where $L:\Q$ is a finite field extension and $\lambda\in L$ ($\OO_L$ is the ring of integers in $L$) and that $t$ acts on $B$ by multiplication with $\lambda$.}
At this point we have $G=\langle t\rangle\ltimes B$ with $B\otimes\Q$ irreducible as  $\Q\langle t\rangle$-module. Therefore for any $z\in B\otimes\Q$,  $B\otimes\Q=\Q\langle t\rangle z$ and the annihilator $I$ of $z$ in $\Q\langle t\rangle$ is generated by an irreducible monic polynomial $p(x)\in\Q[x]$ which is precisely the minimal polynomial of the matrix $M$ representing the action of $t$ on $B$. We choose $z=a_1\otimes{1/m}$ for $a_1$ one of the generators of $A$ and a suitable $m\in\Z$ that we determine below.
 Let $\lambda$ be a root of $p(t)$ and  $L=\Q(\lambda)\cong\Q\langle t\rangle/I$. Then $\Q\subseteq L$ is a finite field extension and we have an isomorphism
$B\otimes\Q\buildrel\cong\over\rightarrow L$ mapping $z$ to $1\in L$ and $z^t$ to $\lambda$. So we have an embedding of $B$ into $L$ so that $a_1$ is mapped to $m$. For each of the other generators $a_i$ with $i=2,\ldots,n$ there is some polynomial  $q_i(t)\in\Z\langle t\rangle$ and some integer $n_i$ with $a_i^{n_i}=a_1^{q_i(t)}$, so if we take as $m$ the smallest common multiple of all those $n_i$'s we deduce that $a_i$ is mapped to $\tilde q_i(\lambda)$ for some integer polynomial $\tilde q_i(t)$, in other words, we have that the image of $A$ lies inside $\OO_L$. We also use $A$ and $B$ to denote its images in $L$.
 We obviously have $B=A[\lambda^{\pm1}]\subseteq \hat B=\OO_L[\lambda^{\pm1}]$. As both $A\subseteq\OO_L$ are free abelian subgroups of the same (finite) rank in $L$, the index of $A$ in $\OO_L$ is finite. From this one deduces that also $B$ has finite index in $\hat B$ and the same happens for $G$  in $\hat G=\langle t\rangle\ltimes\hat B$. This means that we may assume $G=\hat G$. Note that we may also assume that $A=\OO_L$ and that the basis $\{a_1,\ldots,a_n\}$ corresponds to a $\Z$-basis of $\OO_L$, in particular that $a_1$ corresponds to 1.

\subsubsection{We may assume that $\lambda=x/y$ with $x,y\in\OO_L$ such that $\OO_L=x\OO_L+y\OO_L$.}
As $\lambda\in L$ which is the field of fractions of $\OO_L$, we deduce that $\lambda=x_1/y_1$ for some $x_1,y_1\in\OO_L$.
 By Lemma \ref{tech2} below, if $r$ is the ideal class number of $\OO_L$, then $\lambda^r=x/y$ for some $x,y\in\OO_L$ such that $x\OO_L+y\OO_L=\OO_L$. Finally we observe that the group $\tilde G=\langle t^r\rangle\rtimes B$ has finite index in $G$.

 \subsection{The lower bound for elements in $B$: final proof}

\begin{theorem}\label{lowerfin} Let $L:\Q$ be a finite field extension, $\lambda=x/y\in L$ with $x,y\in\OO_L$ such that $\OO_L=x\OO_L+y\OO_L$ and $G=\langle t\rangle \ltimes B$ with $B=\OO_L[\lambda^{\pm1}]$ and $t$ acting on $B$ by multiplication with $\lambda$. Then for any $b\in B$,
$$\|b\|_G\preceq\ln(\mu_E(b)+1).$$
\end{theorem}
\begin{proof}  We let $A=\OO_L$, $1=a_1,\ldots,a_n$ an integer basis of $\OO_L$ and observe that out group is generated by $\mathcal{X}=\{t^{\pm1},a_1^{\pm},\ldots,a_n^{\pm}\}$. We follow the same proof that we have already seen in Proposition \ref{rank1B}. We assume $b\neq 0$. As $b\in B$, there are some $0\leq s,k\in\Z$ so that
$$b\in{1\over x^sy^k}\OO_L$$
thus $d:=bx^sy^k\in\OO_L$ and we may take $s$ and $k$ each smallest possible satisfying that. Therefore  $x,y\nmid d$. 

We use Lemma \ref{tech5} to find $d_1,d_2\in\OO_L$ such that for certain $K$ depending on $x$ and $y$ only,
$$d=x^sd_1+y^kd_2,$$
\begin{equation}\label{normd1d2}\|d_1\|,\|d_2\|\leq \|d\|K^{s+k}.\end{equation}
 Then
$$b={d\over x^sy^k}={d_1\over y^k}+{d_2\over x^s}.$$
We use again Lemma \ref{tech5} to find $r_1,c_1,r_2,c_2\in\OO_L$ such that
$$d_1=x^kr_1+y^kc_1,$$
$$d_2=y^sr_2+x^sc_2$$
such that (here we are taking (\ref{normd1d2}) into account)
$$\label{norma1c1}\begin{aligned}
\|r_1\|,\|c_1\|\leq\|d_1\|K^{2k}\leq\|d\|K^{s+3k}\leq\|d\|(K^3)^{s+k}\\
\|r_2\|,\|c_2\|\leq\|d_2\|K^{2s}\leq\|d\|K^{3s+k}\leq\|d\|(K^3)^{s+k}.\end{aligned}$$
thus 
$$\ln(\|r_1\|+1),\ln(\|c_1\|+1),\ln(\|r_2\|+1),\ln(\|c_2\|+1)\preceq s+k+\ln(\|d\|+1).$$
 We have
$$b={d_1\over y^k}+{d_2\over x^s}=\lambda^kr_1+c_1+c_2+\lambda^{-s}r_2$$
which written multiplicatively is
$$(a_1^{r_1})^{t^k}a_1^{c_1}a_1^{c_2}(a_1^{r_2})^{t^{-s}}$$
so using the bound for elements in $A$ that we have seen in Proposition \ref{rankarbA} we get
$$\begin{aligned}\|b\|_G\leq 2k+2s+\|r_1\|_G+\|c_1\|_G+\|c_2\|_G+\|r_2\|_G\\
\preceq s+k+\ln(\|r_1\|+1)+\ln(\|c_1\|+1)+\ln(\|c_2\|+1)+\ln(\|r_2\|+1)\\
 \preceq s+k+\ln\|d\|\preceq\text{max}\{s,k\}+\ln(\|d\|+1).\end{aligned}$$
 Finally, we note that  the choice of $s,k$ implies that  the hypothesis of
 Lemma \ref{tech6} hold true thus 
$$\text{max}\{s,k\}+\ln\|d\|\preceq\ln(\mu(b)_E+1).$$

\end{proof}

\subsection{Technical Lemmas}
In this section we fix the following notation: $L:\Q$ is a field extension of degree $n$ and $\mathcal{O}_L$ is the ring of integers in $L$.  We fix also an integer basis 
 $\{1=e_1,\ldots,e_n\}$ of $L$ over $\Q$. Associated to this integer basis we have an isomorphism $\pi:L\to \Q^n$ mapping $\OO_L$ onto  $\Z^n$.  The ring $\OO_L$ is a Dedekind domain so although we can not use ordinary factorization as we did when we dealt with the case of Pr\"ufer rank 1, we may use factorization of ideals in $\OO_L$.  We may also speak of fractionary ideals of $L$ and use the ordinary divisibility notation when talking about them.
We denote by $r$ its ideal class number of $\OO_L$, recall that this means that for any ideal $I$ of $\OO_L$, $I^r$ is a principal ideal.

\begin{lemma}\label{tech1} Let $I,J$ be ideals of $\mathcal{O}_L$ such that $I+J=\mathcal{O}_L$. Then for any $s,k\geq 1$,
$I^k+J^s=\mathcal{O}_L$.
\end{lemma}
\begin{proof} If $k\geq 2$, assume by induction that $I^{k-1}+J^{s}=\mathcal{O}_L$. Then
$$\mathcal{O}_L=I^{k-1}(I+J)+J^{s}=I^k+I^{k-1}J+J^s.$$
Now, $I^{k-1}J=I^{k-1}J(I^{k-1}+J^{s})\subseteq I^k+J^s$.
\end{proof}

\begin{lemma}\label{tech2} Let $\lambda\in L$. Then $\lambda^r={x/ y}$ for some $x,y\in\mathcal{O}_L$ such that
$x\mathcal{O}_L+y\OO_L=\OO_L.$
\end{lemma} 
\begin{proof} Put $\lambda={x_1/y_1}$ with $x_1,y_1\in\OO_L$ and let $I=x_1\OO_L+y_1\OO_L$ so that $I$ is the greatest common divisor of the ideals $x_1\OO_L$ and $y_1\OO_L$. Therefore $I\mid x_1\OO_L,y_1\OO_L$ so  $J=x_1\OO_L/I$ and$T=y_1\OO_L/I$ are ordinary coprime ideals in $\OO_L$ thus $\OO_L=J+T$.
 Looking at the fractionary ideal $\lambda\OO_L$ we have
 $$\lambda\OO_L=(x_1\OO_L)/(y_1\OO_L)=J/T.$$ Moreover there are $x',y\in\OO_L$ such that $J^r=x'\OO_L$ and $T^r=y\OO_L$. Thus $\lambda^r\OO_L=J^r/T^r={x'\over y}\OO_L$ and $\lambda^r={x'\over y}u$ for some unit $u\in\OO_L$. Set $x=x'u$. Then $\lambda^r={x/y}$. As $J+T=\OO_L$, by Lemma \ref{tech1} we have 
$$x\OO_L+y\OO_L=J^r+T^r=\OO_L.$$
\end{proof}

\begin{lemma}\label{tech3} There is a constant $C_1\geq 1$ such that for any $u,v\in L$,
$$\|uv\|\leq C_1\|u\|\cdot\|v\|.$$ 
\end{lemma}
\begin{proof} Recall that we had an integer basis $\{e_1,\ldots,e_n\}$ and an associated isomorphism $\pi:L\to\Q^n$. Here we see the elements of $\Q^n$ as column vectors. Let $\pi_l:L\to\Q$ be the composition of $\pi$ with the projection onto the $l$-th coordinate. Let
$D$ be the $\OO_L$-matrix with $e_ie_j$ in the entry $(i,j)$ and for $l=1,\ldots,n$ let 
$$D_l=(\pi_l(e_ie_j))$$
(it is a rational $n\times n$-matrix). For any $u,v\in L$ we have
$$uv=\pi(u)^tD\pi(v)$$
where the overscript $t$ means transpose. If we set 
$w_l=\pi(u)^tD_l,$ then
$$|\pi_l(uv)|=|\pi(u)^tD_l\pi(v)|=|<w_l,\pi(v)>|\leq \|w_l\|\|v\|\leq\|u\|\|D_l^t\|\|v\|.$$

Therefore
$$\|uv\|=\sum_{i=1}^n|\pi_l(uv)|\leq\|u\|\|v\|\sum_{i=1}^n\|D_l^t\|$$ 
and the claim follows with
$$C_1=\sum_{l=1}^n\|D_l^t\|.$$
\end{proof}

\begin{lemma}\label{tech5}  Let $x,y\in\OO_L$ with $x\OO_L+y\OO_L=\OO_L$. There is a constant $K$ depending on $x,y$ such that for any $d\in\OO_L$ and any $s,k\geq 0$ there are $a,c\in\OO_L$ with $d=ax^s+cy^k$ and
$$\|a\|,\|c\|\leq \|d\|K^{s+k}.$$
\end{lemma}
\begin{proof} 
Observe first that Lemma \ref{tech1} implies $x^s\OO_L+y^k\OO_L=\OO_L$. Take $r_1,c_1\in\OO_L$ with $1=r_1x^s+c_1y^k.$ Then 
$${c_1\over x^s}={1\over x^sy^k}-{r_1\over y^k}.$$ 
Let 
$$z={1\over 2x^sy^k}-{r_1\over y^k}\in L$$
and take $w\in \OO_L$ so that $\epsilon:= z-w$ has  smallest possible norm, observe that $\|\epsilon\|\leq {n\over 2}.$ Then
\begin{equation}\label{w+}\begin{aligned}\Big\|w+{r_1\over y^k}\Big\|=\Big\|z+{r_1\over y^k}-\epsilon\Big\|\leq\Big\|z+{r_1\over y^k}\Big\|+\|\epsilon\|\leq{1\over 2}\Big\|{1\over x^sy^k}\Big\|+{n\over 2},\end{aligned}\end{equation}
\begin{equation}\label{w-}\begin{aligned}\Big\|{c_1\over x^s}-w\Big\|=\Big\|{c_1\over x^s}-z+\epsilon\Big\|\leq\Big\|{c_1\over x^s}-z\Big\|+\|\epsilon\|\leq
{1\over 2}\Big\|{1\over x^sy^k}\Big\|+{n\over 2}
\end{aligned}\end{equation}
But by Lemma \ref{tech3} (recall that $C_1\geq 1$)
\begin{equation}\label{secondpart}
\Big\|{1\over x^sy^k}\Big\|\leq  (C_1\Big\|{1\over x}\Big\|)^{s}(C_1\Big\|{1\over y}\Big\|)^{k}\leq K_1^{s+k}
\end{equation}
with $K_1=\text{max}\{C_1\|{1\over x}\|,C_1\|{1\over y}\|,1\}$
Now, let
$$a:=d(r_1+y^kw),$$
$$b:=d(c_1-x^sw).$$
We have
$$d=dx^sr_1+dy^kc_1=dx^sr_1+dy^kc_1+dx^sy^kw-dx^sy^kw=x^sa+y^kb$$
and using again Lemma \ref{tech3},$$
\|a\|=\Big\|dy^k(w+{r_1\over y^k})\Big\|\leq\|d\|\|y\|^k\|w+{r_1\over y^k}\|C_1^{k+1}$$
So taking into account  (\ref{w+}) and (\ref{secondpart}) we deduce
$$\|a\|\leq \|d\|\|y\|^k(K_1^{s+k}+n)C_1^{k+1}.$$

If we use (\ref{w-}) instead of (\ref{w+}) we get an analogous bound for $\|b\|$ and letting $K$ be the largest of the relevant constants we get the result.
\end{proof}

\begin{lemma}\label{tech6} Fix $x,y\in\OO_L$ with $x\OO_L+y\OO_L=\OO_L$. For any $b\in\bigcup_{s,k\geq 0}{1\over x^sy^k}\OO_L$ let  $s,k\geq 0$ be the smallest values so that $d:=x^sy^kb\in\OO_L$.   Then
$$\text{max}\{s,k\}+\ln(\|d\|+1)\preceq \ln(\|b\|_E+1)$$
where the corresponding constants  depend on $x,y$ only.
\end{lemma}
\begin{proof} First,  we may assume that $b$ and $d$ are non-zero. We also assume  that $k\geq s$, for the case when $k\leq s$ one only has to swap $s$ and $k$. Let $F$ be the normal closure of $L$ in $\C$ so that $F:\Q$ is a Galois extension and let  $H=\text{Gal}(F:\Q)$ be its Galois group. For any $a\in F$, we denote by $N(a)=\prod_{\sigma\in H}a^\sigma$ the norm of $a$ in $F$ (in the particular case when $a\in L$, $N(a)$ is a power of the norm of $a$ in $L$).   
We put 
$$\begin{aligned}u:&=\prod_{1\neq\sigma\in H}x^\sigma={N(x)\over x},\\
 v:&=\prod_{1\neq\sigma\in H}y^\sigma={N(y)\over y}\end{aligned}$$ and observe that $u,v\in L$. 
 Let also $c:=du^sv^k$. Note that all the elements $x^\sigma$, $y^\sigma$ are  integers over $\Z$ thus $u,v,c\in\OO_L$ and that
 $$b={d\over x^sy^k}={du^sv^k\over N(x)^sN(y)^k}={c\over N(x)^sN(y)^k}$$
We want to relate $\|d\|$ and $\|c\|$. To do that note that from
$d=(u^{-1})^s(v^{-1})^kc$ we get using Lemma \ref{tech3}
$$\|d\|\leq C_1^{s+k}\|u^{-1}\|^s\|v^{-1}\|^k\|c\|\leq (C_1D)^{s+k}\|c\|$$
for $D=\text{max}\{1,\|u^{-1}\|,\|v^{-1}\|\}$ (the value of $\|u^{-1}\|$ and $\|v^{-1}\|$ is well defined because $u^{-1}$ and $v^{-1}$ lie in $L$). This implies
$$\ln\|d\|\leq (s+k)\ln(C_1D)+\ln\|c\|.$$
The value of $\ln(C_1D)$ could be negative but adding $(s+k)|\ln(C_1D)|$ at both sides we get
$$k+\ln(\|d\|+1)\sim s+k+\ln\|d\|\preceq s+k+\ln\|c\|\sim k+\ln(\|c\|+1).$$
This means that it suffices to prove the assertion in the Lemma for the element $c$ instead of $d$.
Now, recall that we have fixed an isomorphism $\pi:L\to\Q^n$ and denote $\pi(c)=(\gamma_1,\ldots,\gamma_n)\in\Z^n.$  Put also $\beta_i=\gamma_i/N(x)^sN(y)^k$ for $i=1,\ldots,n.$

We claim that  there is some $\epsilon$ depending on $x$ and $y$ only and  a prime integer $p$ such that for some of $\gamma_1,\ldots,\gamma_n$ which we may assume is $\gamma_1$ we have $\gamma_1\neq 0$ and $$p^{\epsilon k}\text{ divides } N(x)^sN(y)^k/\gcd(\gamma_1,N(x)^sN(y)^k).$$
 Assume that the claim is true. Lemma \ref{bound} below implies:
$$\text{If }\gamma_i\neq 0,\, \ln(|\gamma_i|)\leq kC_3+\ln(\mu(\beta_i)),$$
$$k+\ln(|\gamma_1|)\leq C_4\ln(\mu(\beta_1)).$$
Therefore if we set $C_5=(n-1)C_3+1$,
$$\begin{aligned}
k+\ln(\|c\|+1)\sim k+\ln|\gamma_1|+ \sum_{i=2,\gamma_i\neq 0}^n\ln(|\gamma_i|)\leq\\
 kC_5+\ln|\gamma_1|+\sum_{i=2,\beta_i\neq 0}^n\ln(\mu(\beta_i))\leq\\
  (k+\ln|\gamma_1|)C_5+\sum_{i=2,\beta_i\neq 0}^n\ln(\mu(\beta_i))\leq\\
  \ln(\mu(\beta_1)C_4C_5+\sum_{i=2,\beta_i\neq 0}^n\ln(\mu(\beta_i))\preceq
  \ln(\mu_E(b)+1)
  \end{aligned}$$
and we would get the result.  

So we have to prove that the claim above holds true. 
To do that, we let  $P$ be an arbitrary prime ideal of $\OO_L$ and take $p$ the prime integer such that $P$ lies over $p$. We want to compute which is the highest power of $p$ dividing $c$ and $N(x)^sN(y)^k$.  We use  $\nu_P$ to denote the highest power of $P$ dividing a given ideal in $\OO_L$. The ramification index of $p$ in $\OO_L$ is $e_p=\nu_P(p\OO_L)$.  The Galois group $H$ of $F:\Q$ acts on the set of prime factors of $p\OO_L$ and if we
 put $T=\text{Stab}_H(P)$ and take $\{\sigma_1,\dots,\sigma_t\}$ coset representatives of $T$ in $H$, then
 $$p\OO_L=(\prod_{i=1}^tP^{\sigma_i})^{e_p}.$$
We set
$$q_i:=\nu_{P^{\sigma_i}}(x\OO_F),r_i:=\nu_{P^{\sigma_i}}(y\OO_F),$$
$$v_i:=\nu_{P^{\sigma_i}}(d\OO_F),$$
for $i=1,\ldots,t$ and
$$\Lambda_1=\sum_{j=1}^tq_j|T|,\ \Lambda_2=\sum_{j=1}^tr_j|T|.$$
We observe that $x\OO_F$ and $y\OO_F$ are coprime ideals, so if $q_i\neq 0$ for some $i$, then $r_i=0$ and conversely. 
 For $l\geq 0$ we have $p^l\mid c=du^sv^k$ if and only if $(P^{\sigma_i})^{e_pl}\mid c\OO_F$ for $i=1,\ldots,t$. This is equivalent to 
\begin{equation}\label{div}e_pl\leq \nu_{P^{\sigma_i}}(c\OO_F)=\nu_{P^{\sigma_i}}(d\OO_F)+s\nu_{P^{\sigma_i}}(u\OO_F)+k\nu_{P^{\sigma_i}}(v\OO_F)\end{equation}  
 for $i=1,\ldots,t.$ As $u$ and $v$ are a product of $H$-conjugates of $x$ and $y$ resp., we deduce
$$\begin{aligned}\nu_{P^{\sigma_i}}(v\OO_F)=\sum_{1\neq\sigma\in H}\nu_{P^{\sigma_i}}(y^\sigma\OO_F)=\sum_{1\neq\sigma\in H}\nu_{P^{\sigma_i\sigma^{-1}}}(y\OO_F)=\\
\sum_{\sigma_i\neq\sigma\in H}\nu_{P^{\sigma}}(y\OO_F)=
\sum_{j=1}^tr_j|T|-r_i=\Lambda_2-r_i.
\end{aligned}$$ 
and analogously,
$$\nu_{P^{\sigma_i}}(u\OO_F)=\Lambda_1-q_i.$$
Thus (\ref{div}) is equivalent to 
\begin{equation}\label{div2}
e_pl\leq v_i+s\Lambda_1-sq_i+k\Lambda_2-kr_i\ \text{ for }i=1,\ldots,t.
\end{equation}
This means that the highest power of $p$ that divides $c$ in $\OO_F$ is $p^l$ with
$$l=\Big\lfloor{v_j+s\Lambda_1-sq_j+k\Lambda_2-kr_j\over e_p}\Big\rfloor$$
where $j$ is the index so that the value of
$${v_j+s\Lambda_1-sq_j+k\Lambda_2-kr_j\over e_p}$$
is smallest possible (we fix this $j$ from now on). On the other hand, $p^m$  divides $N(x)^sN(y)^k$ if and only if
$$e_pm\leq s\sum_{i=1}^tq_i|T|+ k\sum_{i=1}^tr_i|T|=s\Lambda_1+k\Lambda_2.$$
thus the highest power of $p$ dividing $N(x)^sN(y)^k$ is $p^m$ with
$$m=\Big\lfloor{s\Lambda_1+k\Lambda_2\over e_p}\Big\rfloor.$$
(note that $N(x)$ and $N(y)$ might not be coprime). At this point we are going to fix $\epsilon$. Set
$$\epsilon=\text{min}\{{1\over 2e_p}; p\mid N(x)N(y)\}.$$

The hypothesis that $y\nmid d$ implies that we may choose a  prime ideal $P$ so that, reordering the indices if necessary, $0\leq v_1<r_1$. In particular, this implies $r_1\geq 1$ thus $q_1=0$ and also $r_1-v_1\geq 1$.  
Then
$$\begin{aligned}{s\Lambda_1+k\Lambda_2\over e_p}-{v_j+s\Lambda_1-sq_j+k\Lambda_2-kr_j\over e_p}\geq\\
{s\Lambda_1+k\Lambda_2\over e_p}-{v_1+s\Lambda_1-sq_1+k\Lambda_2-kr_1\over e_p}=\\
{kr_1-v_1\over e_p}={(k-1)r_1+(r_1-v_1)\over e_p}\geq {k\over e_p}.\end{aligned}$$
And from this one easily deduces that
$$m-l=\Big\lfloor {k\Lambda\over e_p}\Big\rfloor-\Big\lfloor {s\Lambda_1-sq_j+k\Lambda_2-kr_j+v_j\over e_p}\Big\rfloor\geq{1\over 2e_p}k\geq\epsilon k.$$

Obviously, as $p^l\mid c$, we have $p^l\mid\gamma_i$ for any $i$. And the fact that $p^l$ is the highest power of $p$ dividing $c$ implies that there must be some of $\gamma_1,\ldots,\gamma_n$ which reordering we may assume is $\gamma_1$ which is not a multiple of $p^{l+1}$ (in particular, $\gamma_1\neq 0$).  Then $p^l$ is the highest power of $p$ dividing $\gcd(\gamma_1,N(x)^sN(y)^k)$  thus $p^{m-l}$ divides $N(x)^sN(y)^k/\gcd(\gamma_1,N(x)^sN(y)^k)$. As $m-l\geq\epsilon k$ we get the claim.
\end{proof}

\begin{lemma}\label{bound} Let $N$, $M\in\Z$ and $\epsilon>0$. 
There are constants $C_3,C_4>0$ depending only on $N$, $M$, $\epsilon$ such that for any $0\neq\gamma\in\Z$, $0\leq s,k\in\Z$ and $\beta:={\gamma/M^sN^k}$ we have
$$\ln|\gamma|\leq C_3\text{max}\{s,k\}+\ln(\mu(\beta))$$
and if there is some prime $p$ such that $p^{\epsilon\text{max}\{s,k\}}$ divides $M^sN^k/\text{gcd}(\gamma,M^sN^k)$, then
$$\text{max}\{s,k\}+\ln|\gamma|\leq C_4\ln(\mu(\beta)).$$
\end{lemma}
\begin{proof} Assume $s\leq k$ the other case being analogous. Let $l=|\text{gcd}(\gamma,M^sN^k)|$. Obviously $l\leq |M^sN^k|\leq |MN|^k$ thus
$$\ln|l|\leq k\ln|MN|.$$
We have $\gamma=l(\gamma/l)$ and $\mu(\beta)=|\gamma/l||M^sN^k/l|$. Therefore
$$\ln|\gamma|=\ln|l|+\ln\Big|{\gamma\over l}\Big|\leq k\ln|MN|+\ln\Big|{\gamma\over l}\Big|+\ln\Big|{M^sN^k\over l}\Big|\leq k\ln|MN|+\ln|\mu(\beta)|$$
so we get the first assertion with $C_3=\ln|MN|$.
Now, assume that there is some prime with 
$$p^{\epsilon k}\mid{M^sN^k\over l}.$$ Then
$\epsilon k\ln|p|\leq\ln|{M^sN^k\over l}|$ and therefore
$$\begin{aligned}
k+\ln|\gamma|=k+\ln|l|+\ln\Big|{\gamma\over l}\Big|\leq k(1+\ln|MN|)+\ln\Big|{\gamma\over l}\Big|\leq\\
 \Big({1+\ln|MN|\over\epsilon\ln|p|}\Big)\ln\Big|{M^sN^k\over l}\Big|+\ln\Big|{\gamma\over l}\Big|\end{aligned}$$
So if we let
$$C_4=\max\{\max\Big\{{1+\ln|MN|\over\epsilon\ln|p|}; p\text{ divides }MN\Big\},1\}$$
we get
$$k+\ln|\gamma|\leq C_4\Big(\ln\Big|{M^sN^k\over l}\Big|+\ln\Big|{\gamma\over l}\Big|\Big)=C_4\ln(\mu(\beta)).$$
\end{proof}

\section{Proofs of the main results}

\noindent{\sl Proof of Theorem \ref{main1}.} If we have $G=Q\ltimes B$ in ii) use Theorem \ref{easyineq} and if we have $B=G'$ use Corollary \ref{easyineq1} to deduce
$$\ln(\mu_E(b)+1)\preceq\|b\|_G.$$
For the converse, note that the group $\hat G$ is as in the hypothesis of Proposition \ref{reduction} so there is some $t\in\hat G$ such that $H=\langle t,B\rangle$ is finitely generated and $t$ acts on $B$ with no eigenvalue of complex norm 1. Moreover the action of $t$ on $B$ is semisimple. Moreover, by Subsection \ref{reductions} be may assume that $H$ is as in the hypothesis of Theorem \ref{lowerfin}. Therefore
$$\|b\|_G\preceq\|b\|_{\hat G}\preceq\|b\|_H\preceq\ln(\mu_E(b)+1).$$
\hfill$\square$

\bigskip

Before we proceed to the proof of Theorem \ref{main1,5} we state and show the result by Kronecker that was mentioned in the introduction.

\begin{lemma}(Kronecker, \cite{Kronecker})\label{Kronecker} Let $q(x)\in\Z[x]$ be an irreducible monic polynomial of degree $m$. If $q(x)$ has some root $\lambda$ with complex norm 1, then $m$ is even and $q(x)=q(1/x)x^m$. 
\end{lemma}
\begin{proof} As $q(x)$ is irreducible, $\lambda\neq\pm 1$ thus $\lambda$ is not a real number. Let $\lambda=a+bi$, $a,b\in\R$. Then $a^2+b^2=1$ which implies
$\lambda^2-2a\lambda+1=0$
thus $$a={\lambda^2+1\over 2\lambda}.$$
In particular, $a\in\Q(\lambda)$ so $a$ is algebraic and $\Q(a)\subseteq\Q(\lambda)$. As $\Q(a)\subseteq\R$ we have $\Q(a)\subsetneq\Q(\lambda)$ and we see that the degree of the extension $\Q(\lambda):\Q(a)$ is precisely 2. Let $h(x)\in\Q[x]$ be an irreducible monic polynomial having $a$ as a root and let $k$ be its degree. The discussion above implies that  $m=2k$. In fact,  $\lambda$ is a root of
$$g(x)=(2x)^kh\Big({x^2+1\over 2x}\Big)$$
which is a polynomial of degree $2k$ thus $g(x)=q(x)$.
Finally,
$$x^{2k}g(1/x)=x^{2k}\Big({2^k\over x^k}\Big)h\Big({{1\over x^2}+1\over{2\over x}}\Big)=g(x).$$
 \end{proof}

 \noindent{\it Proof of Theorem \ref{main1,5}.} As $G$ is abelian-by-cyclic, the hypothesis that $G$ is finitely presented implies that it must be constructible (see for example Corollary 11.4.6 in Lennox and Robinson \cite{LennoxRobinson}). So we deduce that $G$  has  finite Pr\"ufer rank and also, up to swapping $g$ and $g^{-1}$,  that $g$ acts on $B$ via an integer matrix $M$. Let $p(x)\in\Z[x]$ be the minimal polynomial of the matrix $M$ ($p(x)$ is also a generator of the annihilator of $B$ in $\Q\langle g\rangle$). Assume  that  $p(x)$ has some root $\lambda$ of complex norm 1 and choose some $q(x)\in\Z[x]$ irreducible factor  of $p(x)$ having $\lambda$ as root. Observe that there must be some $1\neq b\in B$ with $b^{q(g)}=0$ and from Lemma \ref{Kronecker}, we  deduce that if $m$ is the degree of $q(x)$, $q(x)=q(1/x)x^m$ thus
$b^{q(g^{-1})g^m}=1$ and therefore $b^{q(g^{-1})}=1.$
This implies that the free abelian subgroup of $B$ generated by the finite set $b^g,\ldots,b^{g^{m-1}}$ is setwise invariant under the action of both $g$ and $g^{-1}$ which contradicts the hypothesis. 
\hfill$\square$
\bigskip

Finally, in the proof of Theorem \ref{main2} we will use the next result

\begin{proposition}\label{semidirect} Let $G=Q\ltimes N$ with $Q$ finitely generated and $N$ finitely generated as $Q$-group. Then $G$ is finitely generated and
$$\|g\|_G\thicksim\|q\|_Q+\|n\|_G$$
where $q\in Q$ and $n\in N$ are  elements with $g=qn$ for nontrivial $q$.
\end{proposition}
\begin{proof} Let $S$ be a finite generating system of $Q$ and $T$ a finite generating system of $N$ as $Q$-group. Obviously, $S\cup T$ is a finite generating system of $G$. Let $w$ be a geodesic word in $(S\cup T)^*$ representing $g=qn$.
Let $\rho(w)$ be the number of instances of elements of $S\cup S^{-1}$ in $w$ and let $\nu(w)$ be the number of instances of elements of $T\cup T^{-1}$, so that $\|g\|_G=\rho(w)+\nu(w).$ Let $w_Q$ be the word obtained from $w$ by deleting all the instances of elements of $T\cup T^{-1}$, then  $w_Q=_Gq$ thus $\|q\|_G\leq\rho(w)$. Moreover, if we denote by $w_Q^{-1}$ the inverse word of $w_Q$, then the word $w_Q^{-1}w$ represents an element in  $N$ and since $g=_Gw_Qw_Q^{-1}w$, we see that $n=_Gw_Q^{-1}w$ thus $\|n\|_G\leq 2\rho(w)+\nu(w)$. So we have:
$$3\|g\|_G=3\rho(w)+3\nu(w)\geq\|q\|_G+\|n\|_G.$$
On the other hand, from $g=qn$ we get $\|g\|_G\leq\|q\|_G+\|n\|_G$. 
\end{proof}

\noindent{\sl Proof of Theorem \ref{main2}.} Let $\tilde Q\leq Q$ be a torsion-free finite-index subgroup of $Q$ and choose a basis $q_1\ldots,q_r$ of $\tilde Q$. Choose rational matrices $M_1,\ldots,M_r$ which encode the action of each $q_i$ on $B$ and let $m$ be an integer such that $mM_i$ is integer for $i=1,\ldots,k$. Let
$\theta:G\to B$ be the group homomorphism induced by $\theta(q_i)=q_i$, $\theta(b)=mb$ for $b\in B$. 
 In the proof of  \cite[Theorem 8]{BaumslagBieri}, Baumslag and Bieri show that the HNN-extension
 $$G_1=G\ast_\theta$$ 
is constructible. Then $G_1=Q_1\ltimes B_1$ where $Q$ embeds in $Q_1$ and $B$ embeds in $B_1$, moreover this last group has the same Pr\"ufer rank as $B$ has. Then Theorem \ref{main1} and Proposition \ref{semidirect} imply that for any $g\in G$ such that $g=qb$ for $g=qb$, $q\in Q$, $b\in B$,
$$\|g\|_G\sim\|q\|_Q+\|b\|_G=\|q\|+\ln(\mu_E(b))\sim\|q\|_{Q_1}+\|b\|_{G_1}\sim\|g\|_{G_1}.$$
\hfill$\square$
\goodbreak
\bigskip

\end{document}